\documentclass[reqno, 12pt]{article}

\pdfoutput=1

\usepackage{enumerate}
\usepackage{latexsym}
\usepackage[centertags]{amsmath}
\usepackage{amsfonts}
\usepackage{amssymb}
\usepackage{amsthm}
\usepackage{mathtools}
\usepackage{newlfont}
\usepackage{graphics}
\usepackage{color}
\usepackage{float}
\usepackage{diagbox}
\usepackage{tocloft}
\usepackage{titlesec}
\usepackage{booktabs}
\usepackage{extpfeil}
\usepackage{centernot}
\usepackage[colorlinks=true,linkcolor=blue,citecolor=red,urlcolor=blue]{hyperref} 
\usepackage[linesnumbered,ruled,vlined]{algorithm2e}
\usepackage{url}
\usepackage[T1]{fontenc}
\usepackage{lmodern}
\usepackage{microtype}
\usepackage[nameinlink,noabbrev,capitalize]{cleveref}
\usepackage{longtable}
\usepackage{rotating}
\usepackage{multirow}
\usepackage{extarrows}
\usepackage[sort,compress,numbers]{natbib}
\usepackage[utf8]{inputenc}
\numberwithin{equation}{section}
\usepackage{lmodern}
\usepackage{enumitem}

\newtheorem{theorem}{Theorem}[section]
\newtheorem{thm}[theorem]{Theorem}
\newtheorem{prop}[theorem]{Proposition}
\newtheorem{lem}[theorem]{Lemma}

\allowdisplaybreaks[4]
\SetKwInput{KwInput}{Input}                
\SetKwInput{KwOutput}{Output}              

\title{A polynomial time algorithm for almost bounded denumerant}

\author{Guoce Xin $^a$, Chen Zhang $^{b,*}$, and Zihao Zhang $^c$
\\[2mm]
{\small $^a$ School of Mathematical Sciences, Capital Normal University, }\\[-0.8ex]
{\small Beijing, 100048, PR China}\\
{\small $^b$ Center for Combinatorics, LPMC, Nankai University,}\\[-0.8ex]
{\small Tianjin 300071, PR China}\\
{\small $^c$ School of Mathematics and Statistics, Beijing Institute of Technology,}\\[-0.8ex]
{\small Beijing 102400, PR China}\\
{\small Email addresses: guoce\_xin@163.com (G. Xin), ch\_enz@163.com (C. Zhang),}\\[-0.8ex]
{\small zihao-zhang@foxmail.com (Z. Zhang)}\\[1.2ex]
$^*$ Corresponding author
}

\date{\today}

\begin{document}

\maketitle

\begin{abstract}
Sylvester's denumerant $d(t; \boldsymbol{A})$ counts the number of nonnegative integer solutions to $\sum_{i=1}^{N} a_i x_i = t$, where $\boldsymbol{A} = (a_1, \dots, a_N)$ is a sequence of positive integers with $\gcd(\boldsymbol{A}) = 1$. In 2025, Xin and Zhang gave a polynomial time algorithm in $N$ for computing $d(t; \boldsymbol{A})$ when the entries of $\boldsymbol{A}$ are bounded by a constant. In this paper, we extend this algorithm by incorporating Barvinok's algorithm, enabling it to handle the case where a fixed number of entries of $\boldsymbol{A}$ are allowed to be unbounded.
\end{abstract}

\noindent
\begin{small}
\end{small}

\noindent
\begin{small}
\end{small}

\section{Introduction}

For a positive integer sequence $\boldsymbol{A} = (a_1, a_2, \dots, a_{N})$ with $\gcd(\boldsymbol{A}) = 1$ and a nonnegative integer $t$, the \emph{Sylvester denumerant} \cite{Syl57}, denoted by $d(t; \boldsymbol{A})$, enumerates the number of nonnegative integer solutions to $\sum_{i=1}^{N} a_i x_i = t$. This function has been the subject of extensive research; see, e.g., \cite{Agn02,AL18,BBDDKV15,Bel43,FR02,Lis95,LXZ24,OSu18,SZ12,Uda22}. A classical result states that $d(t; \boldsymbol{A})$ is a quasi-polynomial in $t$ of degree $N-1$, i.e., $d(t; \boldsymbol{A}) = \sum_{i=0}^{N-1} d_i(t; \boldsymbol{A}) t^i$, where each $d_i(t; \boldsymbol{A})$ is a periodic function in $t$.

From the perspective of computational geometry, $d(t; \boldsymbol{A})$ counts the number of integer points in the $(N-1)$-dimensional rational polytope
\[
\mathcal{P}(t; \boldsymbol{A}) = \{\, \mathbf{x} \in \mathbb{R}^N : \boldsymbol{A} \cdot \mathbf{x} = t,\ \mathbf{x} \ge \mathbf{0} \,\}.
\]
For general polytopes, computing $\#(\mathcal{P} \cap \mathbb{Z}^N)$ is $\#P$-hard. Therefore, it is of significant theoretical and practical interest to identify families of polytopes for which this counting problem can be solved in polynomial time. Two prominent such families are known:
\begin{enumerate}
    \item When the dimension $N$ is fixed, Barvinok's algorithm \cite{Bar94,BP99} computes the integer point count in polynomial time.

    \item When the constraint matrix is $\Delta$-modular and the number of constraints is fixed, the algorithm of Gribanov and Zolotykh \cite{GZ22} achieves polynomial time.
\end{enumerate}

However, for the polytope $\mathcal{P}(t; \boldsymbol{A})$, Barvinok's algorithm requires the dimension $N$ to be fixed, while the $\Delta$-modular algorithm requires each $a_i$ to be bounded by a constant $\Delta$. This leaves a significant gap: the case where the dimension is large but the sequence $\boldsymbol{A}$ is \emph{almost bounded}, meaning that a fixed number $k$ of entries are allowed to be unbounded. More precisely, \begin{equation}\label{e-almostboundseq}
\boldsymbol{A} = (a_1, \dots, a_n, b_1, \dots, b_k)
\end{equation} 
where $k = N - n$ is fixed, and $a_i \le C \ll b_j$ for a constant $C$.  

In this paper, we address this case using two algorithms developed by Xin and Zhang in 2025, both accompanied by \texttt{Maple} packages. The first, \texttt{CT-Knapsack} \cite{XZ25Denum}, is designed for computing the top coefficients of $d(t; \boldsymbol{A})$ and relies on Barvinok's algorithm. The second, \texttt{Cyc-Denum} \cite{XZ25Wave}, is a polynomial time algorithm in $N$ when the entries of $\boldsymbol{A}$ are bounded by a constant $C$. It is rooted in the theory of cyclotomic polynomials and leverages results from \cite{XZZ25} on the efficient computation of generalized Todd polynomials. Indeed, the ``almost bounded'' case is beyond the reach of \texttt{Cyc-Denum} and is also challenging for \texttt{CT-Knapsack} when $N$ is large. Our main contribution is to demonstrate that a combination of these two approaches yields a polynomial time algorithm for computing $d(t; \boldsymbol{A})$ when $k$ is fixed. This result adds a new family to the class of polynomial time computable denumerants, bridging the gap between the bounded and fully unbounded cases. Furthermore, it provides an algebraic combinatorial counterpart to the lattice-based methods for Aardal-Lenstra type instances \cite{AL04}.

For a sequence $\boldsymbol{A}$ as in \eqref{e-almostboundseq}, the computation of $d(t; \boldsymbol{A})$ can be reformulated as the following constant term:
\begin{equation}\label{e-denum}
d(t; \boldsymbol{A}) = \mathop{\mathrm{CT}}\limits_{\lambda} \sum_{x_i \ge 0} \lambda^{\sum_{i=1}^{n} a_i x_i + \sum_{j=1}^{k} b_j x_{n+j} - t} = \mathop{\mathrm{CT}}\limits_{\lambda} \frac{\lambda^{-t}}{\prod_{i=1}^{n} (1 - \lambda^{a_i}) \prod_{j=1}^{k} (1 - \lambda^{b_j})},
\end{equation}
where $\mathop{\mathrm{CT}}\limits_{\lambda} f(\lambda)$ denotes the constant term of the Laurent series expansion of $f(\lambda)$ at $\lambda=0$.

Let $F(\lambda)$ be a rational function. We define the operator $L_{\boldsymbol{a}}$ associated with the sequence $\boldsymbol{a} = (a_1, \dots, a_n)$ by
\[
L_{\boldsymbol{a}} F(\lambda) = \mathop{\mathrm{CT}}\limits_{\lambda} \frac{\lambda F(\lambda)}{\prod_{i=1}^{n} (1 - \lambda^{a_i})}.
\]
If $F(\lambda)$ is a power series in $\lambda$, then $L_{\boldsymbol{a}} F(\lambda) = 0$, since the argument of $\mathop{\mathrm{CT}}\limits_\lambda$ has only positive powers. Now set
\[
F(\lambda) = \frac{\lambda^{-t-1}}{\prod_{j=1}^{k} (1 - \lambda^{b_j})}.
\]
Consider the following partial fraction decomposition:
\begin{equation}\label{e-PFD}
F(\lambda) = P(\lambda^{-1}) + R(\lambda)=P(\lambda^{-1}) + \frac{B(\lambda)}{\prod_{j=1}^{k} (1 - \lambda^{b_j})},
\end{equation}
where $P(\lambda^{-1})$ is a Laurent polynomial in $\lambda$ with only negative powers, and $B(\lambda)$ is a polynomial of degree less than $b_1+\cdots+b_k$.
Thus $R(\lambda)$ is a proper rational function and is a power series in $\lambda$. Consequently, $L_{\boldsymbol{a}}( R(\lambda)) = 0$. Combining this with \eqref{e-denum} gives the key reduction
\begin{equation}\label{e-denumLa1}
d(t; \boldsymbol{A}) = L_{\boldsymbol{a}} ( F(\lambda) ) = L_{\boldsymbol{a}} ( P(\lambda^{-1})).
\end{equation}
Thus, the problem reduces to computing the action of $L_{\boldsymbol{a}}$ on the Laurent polynomial part $P(\lambda^{-1})$ obtained from the partial fraction decomposition of $F(\lambda)$.

However, direct computation of $P(\lambda^{-1})$ is impractical when $t$ is large. We convert $P(\lambda^{-1})$ to a constant term, which is indeed another denumerant, but parameterized. This allows us to compute $d(t; \boldsymbol{A})$ by the following algorithm in three steps:

Algorithm \texttt{AlmostBDenum} (sketch)

\begin{enumerate}
  \item[(S1)] Use Barvinok's algorithm to express $\lambda P(\lambda^{-1})$ as a ``short'' sum of rational functions. See Subsection~\ref{ss-barvinok} for details.
  \item[(S2)] Apply the \texttt{CTGTodd} algorithm of Xin et al. \cite{XZZ25} to evaluate the limit for each summand. The result remains a short sum of rational functions.
  \item[(S3)] For each summand compute the residue using technique from Algorithm \texttt{Cyc-Denum}.
\end{enumerate}

The complexities of Steps (S1), (S2), and (S3) are established in Theorem~\ref{thm-Bar}, Proposition~\ref{prop-CTGToddComplexity}, and the analysis in Subsection~\ref{ss-Wf}, respectively. Theorem~\ref{thm-Bar} and Proposition~\ref{prop-CTGToddComplexity} are polynomial for fixed $k$, while the complexity of Step (S3) is independent of $k$ and depends only on the bounded part $\boldsymbol{a}$. Together, these results yield a polynomial time algorithm for computing $d(t; \boldsymbol{A})$ whenever $k$ is fixed.

The implementation of our method combines components from \texttt{CT-Knapsack} (specifically, the Barvinok part) and \texttt{Cyc-Denum}. Therefore, we do not include new experimental data here. Extensive benchmarks demonstrating the efficiency of these constituent algorithms are already available in the original papers \cite{XZ25Denum,XZ25Wave}.

The rest of this paper is organized as follows. In Section~\ref{s-preli}, we review the necessary concepts and tools. In Section~\ref{s-denum}, we derive a rational form formula for $P(\lambda^{-1})$ using Barvinok's algorithm and then compute $d(t; \boldsymbol{A})$ via Algorithm \texttt{Cyc-Denum}. Concluding remarks are given in Section~\ref{s-cr}.

\section{Preliminaries}\label{s-preli}

This section briefly introduces the tools used in our computation scheme for $d(t; \boldsymbol{A})$: (i) residue computations that connect $d(t; \boldsymbol{A})$ to generalized Todd polynomials; (ii) Barvinok's algorithm for short rational generating functions; and (iii) the log-exponential trick for efficient evaluation of generalized Todd polynomials.

\subsection{Results on residues}

A rational function $f(\lambda)$ has a unique Laurent series expansion at any $\lambda_0 \in \mathbb{C}$. We denote by $\mathop{\mathrm{Res}}\limits_{\lambda=\lambda_0} f(\lambda)$ the coefficient of $(\lambda-\lambda_0)^{-1}$ in the expansion. At $\lambda=0$, the series expansion is in $\mathbb{C}((\lambda))$, the field of Laurent series. The following relation holds:
\[
\mathop{\mathrm{CT}}\limits_{\lambda} \lambda f(\lambda) = \mathop{\mathrm{Res}}\limits_{\lambda=0} f(\lambda).
\]

The following two lemmas on residues are needed.
\begin{lem}[\cite{Jac30}]\label{lem-ResMeromor}
Let $c$ be a complex number. Suppose $g(s)$ is holomorphic in a neighborhood of $s=c$ and $f(\lambda)$ is meromorphic in a neighborhood of $\lambda=g(c)$. If $g^{\prime}(c)\neq 0$, then
$$\mathop{\mathrm{Res}}\limits_{\lambda=g(c)}f(\lambda)=\mathop{\mathrm{Res}}\limits_{s=c}f(g(s))g^{\prime}(s).$$
\end{lem}

\begin{lem}[{\cite{LXZ24}}]\label{lem-ResNonegit}
Let $r_1,r_2, \dots, r_k$ be positive integers, and $b \le r_1 + \cdots + r_k - 1$. Suppose
\[
f(\lambda)=\frac{\lambda^{b-1}}{(\lambda-\xi_1)^{r_1}\cdots (\lambda-\xi_k)^{r_k}}.
\]
Then
\[
\mathop{\mathrm{Res}}\limits_{\lambda=0} f(\lambda) = -\sum_{i=1}^k \mathop{\mathrm{Res}}\limits_{\lambda=\xi_i} f(\lambda).
\]
\end{lem}

\subsection{Barvinok's algorithm}\label{ss-barvinok}

A \emph{rational polyhedron} $\mathcal{P} \subset \mathbb{R}^d$ is the set of solutions of a finite system of linear inequalities with integer coefficients:
\[
\mathcal{P} = \{ x \in \mathbb{R}^d : \langle c_i, x \rangle \le \beta_i \text{ for } i = 1, \dots, m\},
\]
where $c_i \in \mathbb{Z}^d$ and $\beta_i \in \mathbb{Z}$. A bounded rational polyhedron is called a \emph{polytope}. We denote its lattice point generating function by
\[
\sigma_{\mathcal{P}}(\boldsymbol{y})=\sum_{\alpha \in \mathcal{P} \cap \mathbb{Z}^d } \boldsymbol{y}^\alpha =\sum_{\alpha \in \mathcal{P} \cap \mathbb{Z}^d } y_1^{\alpha_1}\cdots y_d^{\alpha_d}.
\]
Barvinok \cite{Bar94} proved a seminal result: in fixed dimension, $\sigma_{\mathcal{P}}(\boldsymbol{y})$ can be computed in polynomial time and expressed as a short sum of rational functions. The following theorem, as presented in \cite{BP99}, is the form we will use.

\begin{thm}[{\cite{Bar94,BP99}}]\label{thm-Bar}
Fix $d$. There exists a polynomial time algorithm that, for a given rational polyhedron $\mathcal{P} \subset \mathbb{R}^d$, computes $\sigma_{\mathcal{P}}(\boldsymbol{y})$ as
\[
\sigma_{\mathcal{P}}(\boldsymbol{y}) = \sum_{i\in I} \epsilon_i \frac{\boldsymbol{y}^{\alpha_{i0}}}{\prod_{j=1}^{d} (1-\boldsymbol{y}^{\alpha_{ij}})},
\]
where $\epsilon_i \in \{-1, 1\}$, $\alpha_{i0} \in \mathbb{Z}^d$, and $\alpha_{i1}, \dots, \alpha_{id}$ form a basis of $\mathbb{Z}^d$ for each $i$. The sum is ``short'', meaning the number $|I|$ of summands is bounded by a polynomial in the input size of $\mathcal{P}$.
\end{thm}

\subsection{The log-exponential trick}\label{ss-logexp}

The log-exponential trick, introduced by Xin et al. in \cite{XZZ25}, is a key tool for computing generalized Todd polynomials efficiently. Consider $a \in \mathbb{Q}$ and finite multi-sets of nonzero integers $B_0, \bar B_0, B_1, \bar B_1, \dots, B_r, \bar B_r$. The \emph{generalized Todd polynomials} $gtd_d$ are defined by the generating function
\begin{equation}\label{e-ToddPoly}
F(s) = \sum_{d\ge 0} gtd_d s^d = \mathrm{e}^{as} \frac{\prod_{b \in B_0} \bar h(bs)}{\prod_{b \in \bar B_0} \bar h(bs)} \prod_{i=1}^{r}\frac{\prod_{b \in B_i} \bar h(bs, y_i)}{\prod_{b \in \bar B_i} \bar h(bs, y_i)},
\end{equation}
where
\[
\bar h(s) = \frac{s}{\mathrm{e}^s - 1} = 1 + O(s), \qquad \bar h(s, y) = \frac{1}{1 - y(\mathrm{e}^s-1)} = 1 + O(s).
\]

The trick proceeds by first computing $H(s) = \ln(F(s))$ using the expansions
\[
\ln \bar h(s) = -\sum_{k \ge 1} \frac{\mathcal{B}_k}{k \cdot k!} s^k, \qquad
\ln \bar h(s, y) = \sum_{k \ge 1} C_k(y) s^k,
\]
where $\mathcal{B}_k$ are the Bernoulli numbers and $C_k(y)$ are polynomials in $y$, and then computing $\mathrm{e}^{H(s)}$. A detailed complexity analysis of this procedure, treating the $y_i$'s as variables, was provided in \cite{XZZ25}.

\begin{thm}[{\cite{XZZ25}}]\label{thm-GTodd1}
Suppose $R$ is either $\mathbb{Q}$ or $\mathbb{Z}_p$ with prime $p>d$. Given $a, B_0, \bar B_0, B_1, \bar B_1, \dots$, $B_r, \bar B_r$ defining $F(s)$ as in \eqref{e-ToddPoly}, if $r \ge 1$, then $F(s) \pmod{\langle s^d \rangle}$ (equivalently, the sequence $(gtd_0, \dots, gtd_{d-1})$) can be computed using $O\big((r+1) d^{r+1} \log(d) + \log^2(d) \sum_{i=0}^{r} (|B_i| + |\bar B_i|)\big)$ operations in $R$.
\end{thm}

The case where $\bar B_i$ is empty and $y_i \in R= \mathbb{Q}(\zeta)$ for each $i$, with $\zeta$ a primitive $f$-th root of unity, was considered in \cite{XZ25Wave}.

\begin{thm}[{\cite{XZ25Wave}}]\label{thm-GTodd2}
Let $d$ be a positive integer. For given $a, B_0, B_1, \dots, B_r$, let
\[
F(s) = \sum_{d\ge 0} gtd_d s^d = \mathrm{e}^{as} \prod_{b \in B_0} \bar h(bs) \prod_{i=1}^{r}\prod_{b \in B_i} \bar h(bs, y_i),
\]
where $y_i \in R$. Then we can compute $F(s) \pmod{\langle s^d \rangle}$ using $O\big((r+1) d \log(d) + \log^2(d) \sum_{i=0}^{r} |B_i|\big)$ operations in $R$.
\end{thm}

\section{The computation scheme for $d(t; \boldsymbol{A})$}\label{s-denum}

Recall the reduction $d(t; \boldsymbol{A}) = L_{\boldsymbol{a}}(P(\lambda^{-1}))$ from \eqref{e-denumLa1}. Let
\[
h(\lambda) := \lambda P(\lambda^{-1}),
\]
which is a Laurent polynomial in $\lambda$ with only nonpositive powers. Then
\[
d(t; \boldsymbol{A}) = \mathop{\mathrm{CT}}\limits_{\lambda} \frac{h(\lambda)}{\prod_{i=1}^{n} (1 - \lambda^{a_i})}
= \mathop{\mathrm{Res}}\limits_{\lambda=0} \frac{\lambda^{-1} h(\lambda)}{\prod_{i=1}^{n} (1 - \lambda^{a_i})}
= -\sum_{\zeta} \mathop{\mathrm{Res}}\limits_{\lambda = \zeta} \frac{\lambda^{-1} h(\lambda)}{\prod_{i=1}^{n} (1 - \lambda^{a_i})},
\]
where the sum is over $\zeta \in \Theta:=\{\zeta : \zeta^{a_i} = 1 \text{ for some } i\}$, and the last equality follows from Lemma~\ref{lem-ResNonegit}. By setting $\lambda = \zeta \mathrm{e}^s$ and applying Lemma~\ref{lem-ResMeromor}, we obtain
\[
d(t; \boldsymbol{A}) = -\sum_{\zeta} \mathop{\mathrm{Res}}\limits_{s = 0} \frac{\zeta \mathrm{e}^s (\zeta \mathrm{e}^s)^{-1} h(\zeta \mathrm{e}^s)}{\prod_{i=1}^{n} (1 - \zeta^{a_i} \mathrm{e}^{a_i s})} = - \sum_{f} W_f(t; \boldsymbol{A}),
\]
where $f$ ranges over all divisors of the $a_i$'s, and
\[
W_f(t; \boldsymbol{A}) = \mathop{\mathrm{Res}}\limits_{s = 0}\sum_{\zeta \in \Theta_f} \frac{h(\zeta \mathrm{e}^s)}{\prod_{i=1}^{n} (1 - \zeta^{a_i} \mathrm{e}^{a_i s})}
\]
with $\Theta_f$ denoting the set of primitive $f$-th roots of unity.

The notation $W_f(t; \boldsymbol{A})$ is borrowed from the theory of Sylvester waves (see, e.g., \cite{Uda22,XZ25Wave}). For a classical Sylvester wave, $h(\lambda)$ is a monomial, whereas in this paper it is a short rational encoding of a Laurent polynomial obtained from Steps (S1) and (S2).

\subsection{A short rational encoding of $h(\lambda)$}\label{ss-plambda}
Recall that $h(\lambda) = \lambda P(\lambda^{-1})$. Its short rational encoding is derived as follows.
\begin{prop}\label{prop-plambda}
Let $P(\lambda^{-1})$ be as in \eqref{e-PFD}. Then the following type of formula can be computed in polynomial time.
\[
\lambda P(\lambda^{-1}) = \sum_{i \in I} g_i (\boldsymbol{y})\Big|_{\boldsymbol{z}=\mathbf{1}} = \sum_{i \in I} \epsilon_i \frac{\boldsymbol{y}^{\alpha_{i0}}}{\prod_{j=1}^{k+1} (1 - \boldsymbol{y}^{\alpha_{ij}})} \Big|_{\boldsymbol{z}=\mathbf{1}},
\]
where $\boldsymbol{y} =(\lambda^{-1}, \boldsymbol{z})= (\lambda^{-1}, z_1, \dots, z_k)$, $\epsilon_i \in \{-1, 1\}$, $\alpha_{i0} \in \mathbb{Z}^{k+1}$, $\alpha_{i1}, \dots, \alpha_{i,k+1}$ is a basis of $\mathbb{Z}^{k+1}$ for each $i$, and size $|I|$ of the index set $I$ is bounded by a polynomial in the input size $\sum \log b_j $.
\end{prop}
\begin{proof}
Let $z_0 = \lambda^{-1}$. Direct computation by $[\lambda^{-\ell}] P(\lambda^{-1})= [\lambda^{-\ell}] F(\lambda)$ yields
\begin{align*}
z_0^{-1} P(z_0) &= \sum_{\ell > 0} [\mu^{-\ell}] \frac{\mu^{-t-1}}{\prod_{j=1}^{k} (1 - \mu^{b_j} z_j)} z_0^{\ell - 1} \Big|_{\boldsymbol{z}=\mathbf{1}}
= \sum_{\ell > 0} \mathop{\mathrm{CT}}\limits_{\mu} \frac{\mu^{-t-1} \mu^{\ell} z_0^{\ell-1}}{\prod_{j=1}^{k} (1 - \mu^{b_j} z_j)} \Big|_{\boldsymbol{z}=\mathbf{1}} \\
&= \mathop{\mathrm{CT}}\limits_{\mu} \frac{\mu^{-t} \sum_{\ell > 0} (\mu z_0)^{\ell-1}}{\prod_{j=1}^{k} (1 - \mu^{b_j} z_j)} \Big|_{\boldsymbol{z}=\mathbf{1}}
= \mathop{\mathrm{CT}}\limits_{\mu} \frac{\mu^{-t}}{ (1 - \mu z_0) \prod_{j=1}^{k} (1 - \mu^{b_j} z_j)} \Big|_{\boldsymbol{z}=\mathbf{1}} \\
&= \sum_{x_i \ge 0} \mathop{\mathrm{CT}}\limits_{\mu} \mu^{x_0 + \sum_{j=1}^{k} b_j x_j - t} z_0^{x_0} z_1^{x_1} \cdots z_k^{x_k} \Big|_{\boldsymbol{z}=\mathbf{1}} \\
&= \sigma_{\mathcal{P}}(z_0, z_1, \dots, z_k) \Big|_{\boldsymbol{z}=\mathbf{1}},
\end{align*}
where
\[
  \mathcal{P} = \{ (x_0, x_1, \dots, x_k) \in \mathbb{R}^{k+1} : x_0 + \sum_{j=1}^{k} b_j x_j = t \text{ and } x_0, x_1, \dots, x_k \ge 0 \}
\]
is a rational polytope. The proposition then follows from Theorem~\ref{thm-Bar}.
\end{proof}

By Proposition~\ref{prop-plambda}, we have
\[
h(\lambda) = \lambda P(\lambda^{-1}) = \sum_{i \in I} g_i(\lambda^{-1}, z_1, \dots, z_k) \big|_{\boldsymbol{z}=\mathbf{1}}.
\]

It remains to take the limit at $z_j=1$ for all $j$. This type of limit has been addressed in \cite{Xin15} and discussed further in \cite{XZZ25}. The idea is to first choose an integral vector $(c_1,c_2,\dots, c_k)$ and make the substitution $z_j = \kappa^{c_j}$ such that there are no zeros in the denominator. Now, we need to compute the limit at $\kappa=1$. By letting $\kappa = \mathrm{e}^u$, we can compute separately the constant term of $g_i(\lambda^{-1}, \mathrm{e}^{c_1 u}, \mathrm{e}^{c_2 u}, \dots, \mathrm{e}^{c_k u})$ in $u$. Each $g_i(\lambda^{-1}, \mathrm{e}^{c_1 u}, \mathrm{e}^{c_2 u}, \dots, \mathrm{e}^{c_k u})$ has the following structure:
\[
\frac{\lambda^{m_0} \mathrm{e}^{b_0 u}}{\prod_{b \in B_0} (1-\mathrm{e}^{bu})} \prod_{j=1}^{r} \frac{1}{\prod_{b \in B_j} (1-\lambda^{m_j} \mathrm{e}^{bu})} \prod_{j=r+1}^{r+r'} \frac{1}{(1-\lambda^{m_j})},
\]
where $m_0, b_0 \in \mathbb{Z}$ , $m_j \in \mathbb{Z} \setminus \{0\}$ for $1\le j \le r+r'$, and $B_j \subset \mathbb{Z} \setminus \{0\}$ are finite multi-sets for $0 \le j \le r$. Xin et al. proposed Algorithm \texttt{CTGTodd} based on the log-exponential trick for computing this type of constant term (see \cite[Section 4]{XZZ25}). The algorithm computes $\mathop{\mathrm{CT}}\limits_u g_i(\lambda^{-1}, \mathrm{e}^{c_1 u}, \mathrm{e}^{c_2 u}, \dots, \mathrm{e}^{c_k u}) \pmod p$ for a suitable prime $p$ with a good complexity result as stated in Theorem~\ref{thm-GTodd1}. Indeed, this reslut also works over $\mathbb{Q}$.

\begin{prop}\label{prop-CTGToddComplexity}
Algorithm \texttt{CTGTodd} correctly computes $\mathop{\mathrm{CT}}\limits_ug_i(\lambda^{-1}, \mathrm{e}^{c_1 u}, \mathrm{e}^{c_2 u}, \dots, \mathrm{e}^{c_k u})$ in time $O((k+2)^{k+3} \log(k+2)+  (k+1)\log^2(k+2))$. The output takes the form of a sum of at most $\binom{|B_0| + r}{r}$ simple rational functions of the form:
\begin{equation}\label{e-CTgi}
\frac{c \lambda^{q_0}}{\prod_{j=1}^{v} (1 - \lambda^{q_j})},
\end{equation}
where $c\in \mathbb{Q}$, $q_0, q_j \in \mathbb{Z}$, and $v \le r' + \sum_{j=0}^{r} |B_j| = k+1$.
\end{prop}
\begin{proof}
The computation of $\mathop{\mathrm{CT}}\limits_u g_i$ is indeed Theorem~\ref{thm-GTodd1} in the case when $d = |B_0| + 1$ and the $\bar B_i$'s are all empty. The complexity is then as stated since $\sum_{j=0}^{r} |B_j| = k + 1 - r' \le k + 1$. The remaining conclusions, including the specific formula in \eqref{e-CTgi}, can be found in \cite[Proposition 4.1]{XZ25Denum}, which is a direct corollary of \cite[Proposition 24]{XZZ25}.
\end{proof}

Fix $k$. The seemingly complicated complexity in Proposition~\ref{prop-CTGToddComplexity} is a constant. Meanwhile, the number of summands in the output is at most $\binom{|B_0| + r}{r} \le \binom{2k+2}{k+1}$. Combining this with Proposition~\ref{prop-plambda} yields that $h(\lambda) = \lambda P(\lambda^{-1})$ can be expressed as a short sum of rational functions.

\subsection{Computation of $W_f(t; \boldsymbol{A})$}\label{ss-Wf}
We now detail the computation of $W_f(t; \boldsymbol{A})$, which is the core of Algorithm \texttt{Cyc-Denum} as described in \cite[Section 2.2]{XZ25Wave}. From Subsection~\ref{ss-plambda}, $h(\lambda) = \lambda P(\lambda^{-1})$ is a short sum of terms of the form $c \lambda^{q_0} / \prod_{j=1}^v (1 - \lambda^{q_j})$, which leads us to assume
\[
h(\lambda) = \sum_{i \in I'} h_i(\lambda) = \sum_{i \in I'}  \frac{c_i \lambda^{q_{i,0}}}{\prod_{j=1}^{v_i} (1 - \lambda^{q_{i,j}})}
\]
with $c_i \in \mathbb{Q}$, $q_{i,0}, q_{i,j} \in \mathbb{Z}$, $v_i \le k + 1$. For a particular $i$, substituting $\lambda = \zeta \mathrm{e}^s$ into $h_i(\lambda)/\prod_{\ell=1}^{n} (1 - \zeta^{a_\ell} \mathrm{e}^{a_\ell s})$ gives a term of the form $c_i \zeta^{q_{i,0}} \mathrm{e}^{q_{i,0} s}$, and the denominator becomes $\prod_{\ell=1}^{n+v_i} (1 - \zeta^{a_\ell} \mathrm{e}^{a_\ell s})$, where $a_{n+j} = q_{i,j}$. Therefore, each summand in $W_f(t; \boldsymbol{A})$ can be written as
\[
\sum_{\zeta \in \Theta_f} \mathop{\mathrm{Res}}\limits_{s=0} \frac{c_i \zeta^{q_{i,0}} \mathrm{e}^{q_{i,0} s}}{\prod_{\ell=1}^{n+v_i} (1 - \zeta^{a_\ell} \mathrm{e}^{a_\ell s})}.
\]

For a fixed $f$ and $i$, let $n'_i = \# \{ \ell : f \mid a_\ell, 1 \le \ell \le n+v_i \}$ be the number of denominator factors whose root of unity order is divisible by $f$. Applying Theorem~\ref{thm-GTodd2} with $d = n'_i$, we can compute
\[
\frac{c_i \mathrm{e}^{q_{i,0} s} s^{n'_i}}{\prod_{\ell: f \mid a_\ell} (1 - \zeta^{a_\ell} \mathrm{e}^{a_\ell s})}
\equiv \sum_{l=0}^{n'_i-1} M_{i,l}(\zeta) s^l \pmod{\langle s^{n'_i} \rangle},
\]
where $M_{i,l}(x) \in \mathbb{Q}[x]/\langle \Phi_f(x) \rangle$ and $\Phi_f(x)$ is the $f$-th cyclotomic polynomial. The factors with $f \nmid a_\ell$ are invertible as power series and are absorbed into the coefficients $M_{i,l}(\zeta)$.

Let $\widehat{\Phi}_f(x) := (x^f-1)/\Phi_f(x)$. By \cite[Theorem 2.9]{XZ25Wave}, if
\[
M_{i,n'_i-1}(x) \widehat{\Phi}_f(x) \Phi'_f(x) \equiv \sum_{j=0}^{f-1} \gamma_{i,j} x^j \pmod{\langle x^f - 1 \rangle},
\]
then
\[
\sum_{\zeta \in \Theta_f} \mathop{\mathrm{Res}}\limits_{s = 0} \frac{c_i \zeta^{q_{i,0}} \mathrm{e}^{q_{i,0} s}}{\prod_{\ell=1}^{n+v_i} (1 - \zeta^{a_\ell} \mathrm{e}^{a_\ell s})} = \gamma_{i, \hat q_{i,0}},
\]
where $\hat q_{i,0} \equiv (- q_{i,0} - 1) \pmod{f}$. Summing over all $i$ gives
\[
W_f(t; \boldsymbol{A}) = \sum_{i \in I'} \gamma_{i,\hat q_{i,0}}.
\]

The complexity of computing a Sylvester wave via Algorithm \texttt{Cyc-Denum} was given in \cite[Theorem 2.11]{XZ25Wave}. In our context, for fixed $k$, the computation of each summand in $W_f(t; \boldsymbol{A})$ involves $O(1)$ operations in $\mathbb{Q}[x]/\langle x^f - 1 \rangle$ and $O\big(f \hat n\log(\hat n) + (n+k) \log^2(\hat n)\big)$ operations in $R=\mathbb{Q}(\zeta_f)$, where $\hat n = k + 1 + \#\{\ell: f \mid a_\ell, 1 \le \ell \le n\}$ is an upper bound of $n'_i$ for all $i \in I'$. Since $f$ is a divisor of some $a_i \le C$, $f$ is also bounded by $C$, and the number of distinct $f$'s is bounded by a function of $C$. Therefore, the total cost of Step (S3) is bounded by a constant (depending only on $C$ and $k$) times a polynomial in $n$ and $\log n$.

We are now ready to state the main result of this paper.

\begin{thm}\label{thm-main}
Let $\boldsymbol{A} = (a_1, \dots, a_n, b_1, \dots, b_k)$ be a sequence of positive integers with $\gcd(\boldsymbol{A})=1$, where $k$ is fixed and $a_i \le C$ for all $1 \le i \le n$. Then the \texttt{AlmostBDenum} algorithm computes $d(t; \boldsymbol{A})$ for any nonnegative integer $t$ in time polynomial in $n$.
\end{thm}
\begin{proof}
We analyze the complexity of the three steps of Algorithm \texttt{AlmostBDenum} described in the introduction.

For Step (S1), Proposition~\ref{prop-plambda} shows that $\lambda P(\lambda^{-1})$ can be expressed as a short sum of rational functions using Barvinok's algorithm. Since the dimension $k+1$ is fixed, Theorem~\ref{thm-Bar} guarantees that the number of summands $|I|$ and the time required are bounded by a polynomial in the input size $\sum \log b_j$.

For Step (S2), Proposition~\ref{prop-CTGToddComplexity} establishes that for each summand from Step (S1), Algorithm \texttt{CTGTodd} computes the required limit in time $O((k+2)^{k+3} \log(k+2)+  (k+1)\log^2(k+2))$, which is a constant for fixed $k$. Moreover, the output is a short sum of at most $\binom{2k+2}{k+1}$ rational functions of the form \eqref{e-CTgi}.

For Step (S3), the analysis above shows that for each $f$ dividing some $a_i$, computing each summand in $W_f(t; \boldsymbol{A})$ via Algorithm \texttt{Cyc-Denum} takes time $O\big(f \hat n\log(\hat n) + (n+k) \log^2(\hat n)\big)$ in $R=\mathbb{Q}(\zeta_f)$, where $\hat n = k + 1 + \#\{\ell: f \mid a_\ell, 1 \le \ell \le n\}$. Since $f \le C$, the number of possible $f$'s is bounded by a function of $C$ alone. Thus the total time for Step (S3) is bounded by a constant (depending only on $C$ and $k$) times a polynomial in $n$ and $\log n$.

Combining the three steps, the total running time is polynomial in $n$ and the input size $\sum \log b_j$ (with degree depending on $k$). This proves the theorem.
\end{proof}

\subsection{Examples}
A representative example is $\boldsymbol{A} = (\boldsymbol{a}, \boldsymbol{b})$, where
\begin{align*}
\boldsymbol{a} &= (25, 90, 93, 60, 142, 50, 123, 175, 8, 106, 174, 172, 137, 77, 187, 144, 129, 198, 77, 110), \\
\boldsymbol{b} &= (12223, 36674, 61119).
\end{align*}
The entries of $\boldsymbol{a}$ are bounded by $200$, while the entries of $\boldsymbol{b}$, borrowed from \cite{AL04}, are much larger than $200$. By our method, it can be computed that
\[
d(989894; \boldsymbol{A}) = 26644354315088501086778109382713098487402609326938915018442
\]
using about $72$ seconds. However, neither Algorithm \texttt{CT-Knapsack} nor \texttt{Cyc-Denum}, when used alone, can easily carry out the computation. Algorithm \texttt{CT-Knapsack} is required to handle the $22$-dimensional cone, whereas Algorithm \texttt{Cyc-Denum} involves computations over the $61119$-th cyclotomic field.

A similar example is $\boldsymbol{A} = (\boldsymbol{a}, \boldsymbol{b})$, where
\begin{align*}
\boldsymbol{a} &= (136, 92, 130, 97, 44, 9, 30, 142, 109, 79, 73, 21, 78, 49, 116, 15, 56), \\
\boldsymbol{b} &= (36682, 61139, 73365).
\end{align*}
For which, we can obtain $d(304665; \boldsymbol{A}) = 246782821042899055681586308100746399071650568$ using about $1$ minute.

Finally, we include an example which the reader can easily verify. Let $\boldsymbol{A} = (\boldsymbol{a}, \boldsymbol{b})$, where $\boldsymbol{a} = (2, 5, 6)$ is bounded by $10$, and $\boldsymbol{b} =(81, 107, 129, 1035)$. It is easy to obtain $d(2026; \boldsymbol{A}) = 6485360$ by any method.

The Maple code for the algorithm \texttt{AlmostBDenum}, and the above three examples, are available at \url{https://pan.baidu.com/s/1griXu6gs-ljGa-LleIHiHA} with passcode \texttt{AlBD}.

\section{Concluding remarks}\label{s-cr}
In this paper, we have developed a polynomial time algorithm for computing the denumerant $d(t; \boldsymbol{A})$ when the sequence $\boldsymbol{A}$ is almost bounded, i.e., when a fixed number of entries are allowed to be unbounded. This result generalizes the fully bounded case (solvable by \texttt{Cyc-Denum} \cite{XZ25Wave}) and also includes the fixed dimension case (solvable by Barvinok's algorithm \cite{Bar94,BP99}) as special limits: when $k=0$ we recover the bounded case, and when $n=0$ we recover the fixed dimension case.

Our approach demonstrates the power of combining Barvinok's geometric decomposition with the algebraic constant term method. The key insight is to isolate the ``large'' parameters and reduce the problem to a short sum of nearly bounded denumerants, which are then efficiently evaluated using generalized Todd polynomials \cite{XZZ25}. This adds a new family of polytopes—those defined by an almost bounded sequence of coefficients in one linear equation—to the class of polynomial time computable integer point counters.

From a broader perspective, this work aligns with the program of identifying and exploiting hidden structure in hard knapsack instances. Recent work by Tang, Xin, and Zhang \cite{TXZ26+} developed polynomial time algorithms for Aardal-Lenstra denumerants using similar constant term techniques, focusing on the two-dimensional lattice structure of coefficients parameterized as $a_i = p_i M + r_i N$. Together, these results point toward a unifying framework for polynomial time denumerant computation based on structural decomposition. The techniques developed here is hopefully extended to multiple equations via the notion of almost $\Delta$-modular matrices, where removing a fixed number of columns yields a $\Delta$-modular matrix. Developing polynomial time algorithms for such polytopes with fixed rows is a promising direction for future work, while the case where the number of unbounded entries is large remains a challenging open problem.

\subsection*{Acknowledgements}
The authors would like to express sincere gratitude for all the suggestions that have improved the presentation of this paper. Guoce Xin was partially supported by the National Natural Science Foundation of China (No. 12571355). Chen Zhang was partially supported by the National Natural Science Foundation of China (No. 12601634), and the Postdoctoral Fellowship Program and China Postdoctoral Science Foundation (No. BX20250066).


\begin{thebibliography}{99}
\bibitem{AL04}
K. Aardal and A. K. Lenstra, \emph{Hard equality constrained integer knapsacks}, Math. Oper. Res. 29 (2004), 724--738.

\bibitem{Agn02}
G. Agnarsson, \emph{On the Sylvester denumerants for general restricted partitions}, Proceedings of the Thirtythird Southeastern International Conference on Combinatorics, Graph Theory and Computing (Boca Raton, FL, 2002) 154 (2002), 49--60.

\bibitem{AL18}
F. Aguil\'{o}--Gost and D. Llena, \emph{Computing denumerants in numerical $3$-semigroups}, Quaest. Math. 41 (2018), 1083--1116.

\bibitem{BBDDKV15}
V. Baldoni, N. Berline, J. A. De Loera, B. E. Dutra, M. K\"oppe, and M. Vergne, \emph{Coefficients of Sylvester's denumerant}, Integers 15 (2015), A11.

\bibitem{Bar94}
A. I. Barvinok, \emph{A polynomial time algorithm for counting integral points in polyhedra when the dimension is fixed}, Math. Oper. Res. 19 (1994), 769--779.

\bibitem{BP99}
A. I. Barvinok and J. E. Pommersheim, An algorithmic theory of lattice points in polyhedra, New Perspectives in Algebraic Combinatorics (Berkeley,CA, 1996--97), Math. Sci. Res. Inst. Publ., 38, Cambridge Univ. Press, Cambridge, 1999, pp. 91--147.

\bibitem{Bel43}
E. T. Bell, \emph{Interpolated denumerants and Lambert series}, Am. J. Math. 65 (1943), 382--386.

\bibitem{FR02}
L. G. Fel and B. Y. Rubinstein, \emph{Sylvester Waves in the Coxeter Groups}, Ramanujan J. 6 (2002), 307--329.

\bibitem{GZ22}
D. V. Gribanov and N. Yu. Zolotykh, \emph{On lattice point counting in {{\(\varDelta\)}}-modular polyhedra}, Optim. Lett. 16 (2022), 1991--2018.

\bibitem{Jac30}
C. G. J. Jacobi, \emph{De resolutione aequationum per series infinitas}, J. Reine Angew. Math. 6 (1830), 257--286.

\bibitem{Lis95}
P. Lison\v{e}k, \emph{Denumerants and their approximations}, J. Combin. Math. Combin. Comput. 18 (1995), 225--232.

\bibitem{LXZ24}
F. Liu, G. Xin, and C. Zhang, \emph{Three simple reduction formulas for the denumerant functions}, Ramanujan J. 65 (2024), 1567--1577.

\bibitem{OSu18}
C. O'Sullivan, \emph{Partitions and Sylvester waves}, Ramanujan J. 47 (2018), 339--381.

\bibitem{Syl57}
J. J. Sylvester, \emph{On the partition of numbers}, Q. J. Math. 1 (1857), 141--152.

\bibitem{SZ12}
A. V. Sills and D. Zeilberger, \emph{Formul{\ae} for the number of partitions of $n$ into at most $m$ parts (using the quasi-polynomial ansatz)}, Adv. Appl. Math. 48 (2012), 640--645.

\bibitem{TXZ26+}
J. Tang, G. Xin, and Z. Zhang, \emph{Polynomial-time evaluation of Aardal-Lenstra denumerants via constant term method}, arXiv:2607.11477, 2026.

\bibitem{Uda22}
N. Uday Kiran, \emph{An algebraic approach to $q$-partial fractions and Sylvester denumerants}, Ramanujan J. 59 (2022), 671--712.

\bibitem{Xin15}
G. Xin, \emph{A Euclid style algorithm for MacMahon's partition analysis}, J. Combin. Theory, Ser. A 131 (2015), 32--60.

\bibitem{XZ25Wave}
G. Xin and C. Zhang, \emph{A polynomial time algorithm for Sylvester waves when entries are bounded}, Adv. Appl. Math. 170 (2025), 102931.

\bibitem{XZ25Denum}
G. Xin and C. Zhang, \emph{An algebraic combinatorial approach to Sylvester's denumerant}, Ramanujan J. 66 (2025), 64.

\bibitem{XZZ25}
G. Xin, Y. Zhang, and Z. Zhang, \emph{Fast evaluation of generalized Todd polynomials: applications to MacMahon's partition analysis and integer programming}, J. Symb. Comput. 133 (2025), 102420.

\end{thebibliography}
\end{document}